\documentclass[12pt]{amsart}
\usepackage{amsmath, amssymb}
\newtheorem{theorem}{\sc Theorem}[section]
\newtheorem{lemma}[theorem]{\sc Lemma}
\newtheorem{proposition}[theorem]{\sc Proposition}

\newtheorem{example}[theorem]{\sc Example}

\newcommand{\ep}{\epsilon}

\newcommand{\pr }{\mathrm{Pr} }

\makeatletter
\@namedef{subjclassname@2020}{\textup{2020} Mathematics Subject Classification}
\makeatother

\usepackage{abstract}

\title[Commuting probability]{Finite groups with high commuting probability for Sylow subgroups}

\author[E. Detomi]{Eloisa Detomi}
\address{Eloisa Detomi: Dipartimento di Matematica \lq\lq Tullio Levi-Civita\rq\rq, Universit\`a degli Studi di Padova, Via Trieste 63, 35121 Padova, Italy} 
\email{eloisa.detomi@unipd.it}

\author[D. Senise]{D\'ebora Senise} 
\address{D\'ebora Senise: Department of Mathematics, University of Brasilia, Brasilia, DF, Brazil}
\email{deborasenise2502@gmail.com}

\author[P. Shumyatsky]{Pavel Shumyatsky} 
\address{Pavel Shumyatsky: Department of Mathematics, University of Brasilia, Brasilia, DF, Brazil}
\email{pavel@unb.br}


\begin{document}

\maketitle

\begin{abstract}
\noindent Given two subsets $X,Y$ of a finite group $G$, we write $\Pr(X,Y)$ for the probability that random elements $x \in X$ and $y \in Y$ commute. If $X,Y$ are subgroups, we denote by $\Pr^*(X,Y)$ the maximum real number $\epsilon$ with the property that for every pair 
of distinct primes $p\in\pi(X)$ and $q\in\pi(Y)$ there is a Sylow 
$p$-subgroup $P$ of $X$ and a Sylow $q$-subgroup $Q$ of $Y$ such that 
$\Pr(P,Q) \geq \epsilon$.

\noindent In this paper we handle, among other things, finite groups $G$ with high probabilities $\pr^*(T,G)$, where $T$ is either a term of the lower central series of $G$ or the generalized Fitting subgroup $F_i^*(G)$. Our main results show that the structure of such groups is similar, in some precise sense, to that of nilpotent groups.
\end{abstract}

\section{Introduction}

Given two subsets $X, Y$ of a finite group $G$, we write $\Pr(X,Y)$ for
the probability that random elements $x \in X$ and $y \in Y$ commute.
The number $\Pr(G,G)$ is called the commuting probability of $G$. There is an extensive recent literature dedicated to this subject.
It is well-known that $\Pr(G,G) \leq 5/8$ for any nonabelian group $G$. 
Another important result is the theorem of P.~M.~Neumann~\cite{neumann1989two}, 
which states that if $G$ is a finite group and $\epsilon$ is a positive number such that
$\Pr(G,G) \geq \epsilon,$ then  $G$ has a normal subgroup  $R$ such that both the index $|G:R|$ and the order of the commutator subgroup $[R,R]$ are $\epsilon$-bounded.
Throughout the article we use the expression ``$(a,b,\ldots)$-bounded''
to mean that a quantity is bounded from above by a number depending only on the parameters $a,b,\ldots$.
A number of further results on the commuting probability in finite
groups can be found in \cite{burness2023,ds22,dlms23,ref_7eberhard2015, erfanian2007, martinez2024}.

Given a finite group $G$ and a prime $p\in\pi(G)$, we write $Syl_p(G)$ for the set of Sylow $p$-subgroups of $G$. If $X,Y$ are subgroups of $G$, we denote by $\Pr^*(X,Y)$ the maximum real number $\epsilon$ with the property that for every pair 
of distinct primes $p\in\pi(X)$ and $q\in\pi(Y)$ there is $P\in Syl_p(X)$ and $Q\in Syl_q(Y)$ such that $\Pr(P,Q) \geq \epsilon$.

It is easy to see that $G$ is nilpotent if and only if $\Pr^*(G,G)=1$, that is, if and only if coprime Sylow subgroups of $G$ commute. 
The paper \cite{dlms23} handles finite groups $G$ in which the condition holds with high probability. It is shown that if $\Pr^*(G,G)\geq\epsilon>0$, then the index $|G:F_2(G)|$ is $\epsilon$-bounded. Moreover, if for any $p\in\pi(G)$ there is $P\in Syl_p(G)$ and a Hall $p'$-subgroup $H\leq G$ such that $\Pr(P,H)\ge\epsilon$, then $|G:F(G)|$ is $\epsilon$-bounded.  Here and throughout $F(G)$ stands for the Fitting subgroup and $F_i(G)$ for the $i$th term of the upper Fitting series of $G$. 

Let $\gamma_{\infty}(G)$ be the intersection of the lower central series of a group $G$. The present paper grew out of the observation that a finite group $G$ is nilpotent if and only if  $\pr^*(\gamma_{\infty}(G),G)=1$. Here we explore finite groups with high commuting probability $\pr^*(\gamma_{\infty}(G),G)$ or $\Pr(P,G)$, where $P$ runs over the set of Sylow subgroups of $\gamma_{\infty}(G)$. Recall that the exponent of a finite group $G$ is the minimum number $e$ such that $x^e=1$ for all $x\in G$.

\begin{theorem} \label{expG/F(G)} Let $\epsilon>0$.   Let $G$ be a finite group and $T = \gamma_{\infty}(G)$. Assume that  $\pr^*(T,G)\ge\epsilon$. Then 
\begin{enumerate}
\item The index $|G:F_2(G)|$ is $\epsilon$-bounded; 
\item The exponent of $G/F(G)$ is $\epsilon$-bounded. 
\end{enumerate}
\end{theorem}

\begin{theorem} \label{Hall1} Let $\epsilon>0$. Let $G$ be a finite soluble group and $T = \gamma_{\infty}(G)$. Suppose that for any prime $p \in \pi(G)$ there is $P\in Syl_p(G)$ and a Hall $p'$-subgroup $H$ of ${T}$ such that $\Pr(H,P)\geq\epsilon$. Then the index $|G:F(G)|$ is $\epsilon$-bounded.
\end{theorem}

Somewhat surprisingly, it is not true that if for each  $p\in \pi(T)$ there is $P\in Syl_p(T)$ and a Hall $p'$-subgroup $H$ of $G$ such that $\Pr(P,H) \geq\epsilon$, then $|G:F(G)|$ is $\epsilon$-bounded (see Example \ref{exem1}). 
 On the other hand, it will be shown that under this hypothesis the exponent of $G/F(G)$ is $\epsilon$-bounded. 

The following result was obtained in \cite{ds22}: Suppose that $\Pr(P,G)\ge\epsilon$ whenever $P$ is a Sylow subgroup of $G$. Then $G$ has a normal subgroup $R$ such that both the index $|G:R|$ and the order of the commutator subgroup $[R,R]$ are $\epsilon$-bounded. Let $k\geq1$, and let $\gamma_k(G)$ be the $k$th term of the lower central series of a group $G$. Here we prove

\begin{theorem} \label{newTh} Let $G$ be a finite group and $T=\gamma_k(G)$ for some $k\ge1$. Suppose that $\Pr(P,G)\ge\epsilon>0$ whenever $P$ is a Sylow subgroup of $T$. Then $G$ has a nilpotent normal subgroup $R$ of nilpotency class at most $k+1$ such that both the exponent of $G/R$ and the order of $\gamma_{k+1}(R)$ are $\epsilon$-bounded.
\end{theorem}

Notably, we  give an example showing that the index $|G:R|$ in Theorem \ref{newTh} cannot be bounded in terms of $\epsilon$ only (see Example \ref{exem1}).
 This came to us as a surprise since at the start of this work we had no reason to expect that the conclusion in Theorem \ref{newTh} should differ from the result in \cite{ds22}.

Recall that the generalized Fitting subgroup $F^*(G)$ is the product of the Fitting subgroup $F(G)$ and all subnormal quasisimple
subgroups; here a group is quasisimple if it is perfect and its quotient by the centre
is a nonabelian simple group. Then the generalized Fitting series of G is defined
starting from $F_1^*(G)=F^*(G)$, and then by induction, with $F_{i+1}^*(G)$ being the inverse image of $F^*(G/F_i^*(G))$. 

\begin{theorem} \label{genfit}   Let $G$ be a finite group and $T = F^*(G)$. Assume that  $\pr^*(T,G)\ge\epsilon>0$ and set $s = |\pi(F(G))| $. Then
\begin{enumerate}
\item The exponent of $G/F(G)$ is $\epsilon$-bounded. 
\item The index  $|G:F(G)|$ is $(s, \epsilon)$-bounded. 
\end{enumerate}

\end{theorem}

We provide an example showing that the index $|G:F^*_2(G)|$ (and therefore also the index $|G:F_2(G)|$) in Theorem \ref{genfit} cannot be bounded in terms of $\epsilon$ only
(see Example \ref{exem2}). 
 On the other hand, we establish the following theorem.

\begin{theorem} \label{genfit2}   Let $G$ be a finite group and $T=F_2^*(G)$. Assume that  $\pr^*(T,G)\ge\epsilon>0$. Then the index $|G:F_2(G)|$ is $\epsilon$-bounded. 
\end{theorem}

Remark that the statements of the above theorems remain true if we drop the hypothesis that $T=\gamma_{\infty}(G)$ (respectively, $T=\gamma_k(G)$, $T=F^*(G)$ or $T=F_2^*(G)$) and instead just assume that $T \ge \gamma_{\infty}(G) $ (respectively, $ T \ge \gamma_k(G) $, $T \ge  F^*(G)$ or $T \ge  F_2^*(G)$). This is because for any subgroups $K\leq H\leq G$ with $K$ normal in $H$ we have $\pr^*(K,G)\ge\pr^*(H,G)$ (see the next section for more details).

\section{Preliminaries}
All groups considered in this paper are finite and this will be tacitly assumed throughout. We slightly abuse the standard terminology and say that a group $G$ has trivial Sylow $p$-subgroup whenever $p$ does not divide $|G|$.

Let a group $A$ act by automorphisms on a group $G$. We say for 
short that the group $A$ acts on a group $G$ coprimely if the orders 
of $A$ and $G$ are coprime, that is, $(|A|,|G|) = 1$. 
The following lemma records some standard results on coprime action. 
In the sequel the lemma will often be used without explicit references.

\begin{lemma}
[{\cite[(24.1),(24.5),(24.6)]{aschbacher2000}}]
\label{lemma:coprime_action}
Let a group $A$ act coprimely on a group $G$ and let $\Phi(G)$ 
be the Frattini subgroup of $G$. Then:
\begin{enumerate}
  \item If $N$ is a normal $A$-invariant subgroup of $G$, then 
        $C_{G/N}(A) = C_G(A)N/N$.
  \item $[G,A,A] = [G,A]$.
  \item If $G$ is abelian, then $G = [G,A] \times C_G(A)$.
  \item If $A$ centralizes $G/\Phi(G)$, then $A$ centralizes $G$.
\end{enumerate}
\end{lemma}
\medskip

We now collect some results on commuting probability that will be needed. 
\medskip

If $X, Y$ are subsets of a     group $G$, we have
$$\Pr(X,Y) = \frac{|\{(x,y) \in X \times Y \mid xy = yx\}|}{|X||Y|}. $$
\medskip

Note that $\Pr(X,Y) = \Pr(Y,X)$ and
\[
\Pr(X,Y) = \frac{1}{|Y|} \sum_{y \in Y} \frac{|C_X(y)|}{|X|}
         = \frac{1}{|X|} \sum_{x \in X} \frac{|C_Y(x)|}{|Y|},
\]
where, as usual, $C_Y(x)$ denotes the set of all elements of $Y$ 
centralizing $x$.

The next lemmas are given without proofs as they can be found in \cite{dlms23}.

\begin{lemma}[Lemma 2.2, \cite{dlms23}] \label{subgroups}
Let $G$ be a group and let $H, K$ be subgroups of $G$. Then
\begin{enumerate}

    \item If $N$ is a normal subgroup of $G$, then $ \Pr(HN/N, KN/N) \geq \Pr(H,K).$
    
    \item If $H_0 \leq H$, then $\Pr(H_0,K) \geq \Pr(H,K).$
    
    \item If $G = G_1 \times G_2$, $H_i \leq G_i$ and $K_i \leq G_i$, then 
    $$\Pr(H_1 \times H_2, K_1 \times K_2) = \Pr(H_1,K_1)\Pr(H_2,K_2).$$
    
\end{enumerate}
\end{lemma}

\begin{lemma}[Lemma 2.8, \cite{dlms23}]  \label{lemma2.8}
Let $G$ be a group and let $H, K$ be subgroups of $G$ with 
$\Pr(H,K)\geq\epsilon>0$. Then there exists a subset $X$ of $H$ such that 
for $H_0 = \langle X \rangle$ the following holds.
\begin{enumerate}
    \item $|H : H_0| \leq 2/\epsilon - 1$;
    \item $|K : C_K(x)| \leq 2/\epsilon$ for every $x \in X$;
    \item $|K : C_K(x)| \leq (2/\epsilon)^{6/\epsilon}$ for every $x \in H_0$.
\end{enumerate}
\end{lemma}

We observe that if in the above lemma $H$ normalizes $K$, then $H$ contains a normal subgroup $H_0$ with the required properties. 

\begin{lemma}[Lemma 2.10, \cite{dlms23}] 
\label{lm2.10}
Let $G$ be a group, $P$ a $p$-subgroup and $Q$ a $q$-subgroup of $G$ for two different primes $p$ and $q$. Assume that $\Pr(P,Q) \ge \epsilon > 0$.

\begin{enumerate}
    \item If $P$ normalizes $Q$ and $p > (2/\epsilon)^{6/\epsilon}$, then $[P,Q] = 1$.
    \item If $p > (2/\epsilon)^{6/\epsilon}$, then $Q$ has a normal subgroup $Q_0$ such that $  |Q : Q_0| \le \lfloor 2/\epsilon \rfloor!$ and $ [P, Q_0] = 1. $ 
\end{enumerate}
\end{lemma}
\medskip

Given subgroups $X,Y\leq G$, recall that $\pr^*(X,Y)$ denotes the maximum real number $\epsilon$ with the property that for every pair of distinct primes $p \in \pi(X)$ and $q \in \pi(Y)$ there is $P \in Syl_p(X)$ and $Q \in Syl_q(Y)$ such that $\pr(P,Q) \ge \epsilon$.

\begin{lemma}[Corollary 4.3, \cite{dlms23}] 
\label{Eb}
Let $G$ be a direct product of nonabelian simple groups such that $\pr^{*}(G, G) \ge \epsilon > 0$. Then $|G|$ is $\epsilon$-bounded.
\end{lemma}

In view of the above results, the following lemma is now straightforward. 

\begin{lemma}
\label{star}
Let $G$ be a group.
\begin{enumerate} 
\item If $K\leq H$ and $L$ are subgroups of $G$, with $K$ normal in $H$, then $\pr^*(K,L)\ge\pr^*(H,L)$;
\item If $N$ is a normal subgroup of $G$, then for any $H\leq G$ we have $\pr^*(HN/N,G/N)\ge\pr^*(H,G)$.
\end{enumerate}
\end{lemma}

The next lemma is Lemma 2.6 from \cite{detomi2021}.

\begin{lemma}
\label{lemma4.4}
Let $G = QH$ be a group with a normal nilpotent subgroup $Q$ 
and a subgroup $H$ such that $(|Q|, |H|) = 1$ and 
$|Q : C_Q(x)| \leq m$ for all $x \in H$. 
Then the order of $[Q,H]$ is $m$-bounded.
\end{lemma}

We will require the following lemma.

\begin{lemma} \label{boundindex}
    Let $G = QH$ be a group with a normal nilpotent subgroup $Q$ 
and a subgroup $H$ such that $(|Q|, |H|) = 1$ and let $T = \gamma_k(G)$ for some $k \ge 1$. Assume that $\pr(Q \cap T,H) \ge \epsilon$. Then $|H: C_H(Q)|$ is $\epsilon$-bounded. 
\end{lemma}

\begin{proof}   
Set $N = Q \cap T $. Since $\pr(N, H) \geq \epsilon$, it follows from Lemma \ref{lemma2.8}  that there is a normal subgroup $H_0$ of $H$ such that $|H : H_0| \leq 2/\epsilon$ and $|N : C_{N}(x)| \leq (2/\epsilon)^{6/\epsilon}$ for every $x \in H_0$. Applying Lemma \ref{lemma4.4} to $NH_0$ we obtain that $[H_0, N]$ has $\epsilon$-bounded order. Moreover, we know that $[ Q, H_0, \overset{k}{\dots} , H_0] \le [ N, H_0]$.
Since $H_0$ acts coprimely on $Q$, it follows that $[Q, H_0, \overset{k}{\dots} , H_0] = [Q, H_0]$ has $\epsilon$-bounded order. 
Thus $C = C_{H_0}([Q, H_0])$ has $\epsilon$-bounded index in $H_0$. Furthermore,
\[
[Q, C, C] \leq [Q, H_0, C] = 1.
\]
As $C$ acs coprimely on $Q$, it follows that $[Q,C] = [Q,C,C]=1$, and therefore 
 $C$ centralizes $Q$. Thus, $|H_0 : C_{H_0}(Q)|$ is $\epsilon$-bounded, and therefore $|H : C_H(Q)|$ is $\epsilon$-bounded as well.
\end{proof}

\section{Theorems 1.1 and 1.2} 

In this section $T$ stands for $\gamma_{\infty}(G)$. 

\begin{lemma}
\label{P<F(G)}
Let $G$ be a soluble group such that $\pr^*(T,G)\geq \epsilon > 0$. If $p > (2/\epsilon)^{6/\epsilon}$, then every Sylow $p$-subgroup of $G$ is contained in $F(G)$.
\end{lemma}

\begin{proof}
    Let $p$ be a prime such that $p > (2/\epsilon)^{6/\epsilon}$, and let $P$ be a Sylow $p$-subgroup of $G$. We need to show that $P \leq F(G)$. Suppose that this is false. Without loss of generality, we may assume that $O_p(G) = 1$, so that $F(G)$ is a $p'$-subgroup. 
Write 
$$F(G) = S_1 \times \cdots \times S_l,$$ where $S_i$ is a Sylow $p_i$-subgroup of $F(G)$. 
Note that $S_i \cap T$ is normal in $G$ and so it is contained in a Sylow $p_i$-subgroup $Q_i$ of $T$ having the property that $ \pr(P, Q_i)\geq\epsilon$. 
It follows from Lemma \ref{subgroups} (2) that $ \pr(P,S_i \cap T)\geq\epsilon$. 
In view of Lemma \ref{lm2.10} (1) we deduce that $[S_i \cap T, P] = 1$. 
Let $c$ be the positive integer such that $\gamma_{\infty}(G) = \gamma_c(G)$. Since $[S_i, P, \overset{c-1}{\dots} , P] \le  S_i \cap T$ this implies
 $$[S_i, P, \overset{c}{\dots} , P] \le [S_i \cap T, P] = 1.$$
As $P$ acs coprimely on $S_i$, it follows that $[S_i,P, \overset{c}{\dots}, P] = [S_i,P]$, and therefore $[S_i,P] = 1$. 
Consequently, $P$ centralizes $F(G)$. A well-known property of finite soluble groups is that the centralizer of $F(G)$ is contained in $F(G)$. Thus, we have a contradiction, which proves the lemma. 
\end{proof}

\begin{proposition}
\label{F2sol}
Let $G$ be a soluble group such that $\pr^*(T,G)\geq \epsilon > 0$. Then $F_2(G)$ has $\epsilon$-bounded index in $G$.
\end{proposition}

\begin{proof}
  By Lemma \ref{P<F(G)}, if $P$ is a Sylow $p$-subgroup of $G$ for a prime $p > (2/\epsilon)^{6/\epsilon}$, then $P$ is contained in $F(G)$. Passing to the quotient over $F(G)$, we assume that the prime divisors of $|G|$ are smaller than $(2/\epsilon)^{6/\epsilon}$. In particular, the cardinality of $\pi(G)$ is $\epsilon$-bounded. We will show that under this assumption $F(G)$ has $\epsilon$-bounded index in $G$. This will imply the desired result.

Let $p$ be a prime divisor of $|G|$ and let $P$ be a Sylow $p$-subgroup of $G$. 
First assume that $O_p(G) = 1$, in which case $F(G)$ is a $p'$-subgroup. Write
$$
F(G) = S_1 \times \cdots \times S_l,
$$
where $S_i$ is a Sylow $p_i$-subgroup of $F(G)$ and $l$ is $\epsilon$-bounded. Set $S_i \cap T = N_i$. For each $p_i$, there exists a Sylow $p_i$-subgroup $Q_i$ of $T$ such that $\pr(P, Q_i) \geq \epsilon$. As $S_i$ is normal in $G$, it follows that $N_i \leq Q_i$, therefore $\pr(P, N_i) \geq \epsilon$ by Lemma \ref{subgroups} (2). Thus,
it follows from Lemma \ref{boundindex} applied to $S_iP$ that $|P : C_P(S_i)|$ is $\epsilon$-bounded.
Since $l$ is $\epsilon$-bounded, we deduce that $|P : C_P(F(G))|$ is $\epsilon$-bounded. As $G$ is soluble, we have that $C_G(F(G)) \leq F(G)$. On the other hand $F(G)$ is a $p'$-group, so $C_P(F(G)) = 1$. Thus, $P$ has $\epsilon$-bounded order.

Now we drop the assumption that $O_p(G) = 1$ and apply the previous argument to $G/O_p(G)$. We deduce that the order of $P/O_p(G)$ is $\epsilon$-bounded. Since there are only $\epsilon$-boundedly many primes dividing $|G|$, it follows that the order of $G/F(G)$ is $\epsilon$-bounded. This concludes the proof. 
\end{proof}

We will now prove Theorem \ref{expG/F(G)}.

\begin{proof}[Proof of Theorem \ref{expG/F(G)}]

Part (1). Let $\epsilon>0$ and $G$ be a group such that $\pr^*(T,G)\geq\epsilon$. We need to show that the index of $F_2(G)$ in $G$ is $\epsilon$-bounded.

Let $R = R(G)$ be the soluble radical of $G$.  By Proposition \ref{F2sol}, the index of $F_2(R)$ in $R$ is $\epsilon$-bounded. As $F_2(R) \leq F_2(G)$, it is sufficient to prove that the order of $G/R$ is $\epsilon$-bounded. By Lemma \ref{star} (2) we have that 
$ \pr^*(TR/R , G/R) \geq \epsilon$, and since $\gamma_{\infty}(G/R) = TR/R $, without loss of generality we may assume that $R=1$. Hence, our aim is to show that the order of $G$ is $\epsilon$-bounded.

 Let $F^* = F^*(G)$ be  the generalized Fitting subgroup. Observe that under our hypotheses $F^*$ coincides with the socle of $G$ and so $F^*\le T$. Thus, by Lemma \ref{star}, 
 $$\pr^*(F^*,F^*) \ge \pr^*(T,G) \ge \epsilon.$$ 
 So we can apply Lemma \ref{Eb} to $F^*$ and see that the order of $F^*$ is $\epsilon$-bounded. Consequently, the index of $C_G(F^*)$ in $G$ is $\epsilon$-bounded as well, and since $C_G(F^*) \leq F^*$, we conclude that the order of $G$ is $\epsilon$-bounded. 

\vspace{0.5cm}
Part (2). Recall that $\Pr^*(T,G) \ge \epsilon > 0$. We need to show that the exponent of $G/F(G)$ is $\epsilon$-bounded. 
    
According to part (1) of the theorem, the subgroup $F_2(G)$ has $\epsilon$-bounded index in $G$; in particular $G/F_2(G)$ has $\epsilon$-bounded exponent. So we just need to prove that $F_2(G)/F(G)$ has $\epsilon$-bounded exponent. Thus, without loss of generality we may assume that $G=F_2(G)$.

Choose a prime $p\in\pi(G)$ and let $P$ be a Sylow $p$-subgroup of $G$. If for every $p \in \pi(G)$ the exponent of $P/O_p(G)$ is $\epsilon$-bounded, say at most $n$, then we must have $P = O_p(G)$ for every $p > n$. Since $n$ is $\epsilon$-bounded and does not depend on $p$, it is sufficient to show that $P/O_p(G)$ has $\epsilon$-bounded exponent. Without loss of generality we may assume that $O_p(G) = 1$. Write 
 $$F(G)=S_1\times \dots \times S_l,$$
where each $S_i$ is the Sylow $p_i$-subgroup of $F(G)$. 

 Set $N_i=S_i \cap T$ for $i = 1, \dots, l $. Note that $N_i$ is normal in $G$ and it is contained in the Sylow $p_i$-subgroup $Q_i$ of $T$ having the property that $\pr(P,Q_i) \ge \epsilon $. It follows from Lemma \ref{subgroups} (2) that $\pr(P, N_i) \ge \epsilon$. In view of Lemma \ref{boundindex} applied to $S_iP$ we deduce that $|P : C_P(S_i)|$ is $\epsilon$-bounded.

Set $p^{m_i}=|P : C_P(S_i)|$ and let $m= max\{m_i, 1 \le i \le l \}$. If $x \in P$, then $x^{p^{m}} \in C_G(F(G))$. 
This implies that $P^{p^m}$ centralizes $F(G)$. Since $G$ is soluble, it follows that $P^{p^m} \le F(G)$. On the other hand, $O_p(G) =1$ and we conclude that $x^{p^{m}} = 1$. Thus, $P$ has $\epsilon$-bounded exponent. The proof is now complete.
\end{proof}

Now we will prove Theorem \ref{Hall1}.

\begin{proof}[Proof of Theorem \ref{Hall1}]
Recall that $G$ is a soluble group and $\epsilon > 0$ is a real number such that for any prime $p\in\pi(G)$ there is a Sylow $p$-subgroup $P$ 
of $G$ and a Hall $p'$-subgroup $H$ of $T$ such that $\pr(H,P)\geq \epsilon$. We need to show that the index of $F(G)$ in $G$ is $\epsilon$-bounded.

We can assume that $G$ is not of prime power order.  Let $P$ be a Sylow $p$-subgroup of $G$. We wish to show that $P \cap F(G)$ has $\epsilon$-bounded index in $P$. Without loss of generality we may assume that $O_p(G) = 1$, in which case $F(G)$ is a $p'$-subgroup contained in any Hall $p'$-subgroup of $G$. 

Note that $ T \cap F(G)$ is normal in $T$ and so it is contained in a Hall $p'$-subgroup $H$ of $T$ having the property that $\pr(H, P) \ge \epsilon$. It follows from Lemma \ref{subgroups} (2) that $\pr(T \cap F(G) , P) \ge \epsilon $. In view of Lemma \ref{boundindex} we deduce that $|P : C_P(F(G))|$ is $\epsilon$-bounded.
As $C_G(F(G)) \leq F(G)$ and $F(G)$ is a $p'$-subgroup, it follows that $C_P(F(G)) = 1$ and so $P$ has $\epsilon$-bounded order. 

Thus, for any $p \in \pi(G)$ and any Sylow $p$-subgroup $P$ of $G$ the index of $P \cap F(G)$ in $P$ is $\epsilon$-bounded, say at most $m$. In particular, $P \leq F(G)$ for all primes $p > m$. Since $m$ does not depend on $p$ and is $\epsilon$-bounded, it follows that the order of $G/F(G)$ is $\epsilon$-bounded, as required.
\end{proof}

We now furnish an example showing that, in a soluble group $G$ such that for any prime $p\in\pi(T)$ there is a Sylow $p$-subgroup $P$ of $T$ and a Hall $p'$-subgroup $H$ of $G$ for which $\pr(P,H)\geq\epsilon$, the index of $F(G)$  in $G$ can be arbitrarily large.

\begin{example}\label{exem1}{\rm
Let $p_1, \ldots, p_s$ be distinct odd primes and $D_i$ the dihedral group of order $2p_i$.
Let $G$ be the direct product of the groups $D_i$.
It is clear that $T = \gamma_2(G)$ is cyclic of order $p_1\cdots p_s$, and whenever $P$ is a Sylow $p$-subgroup of $T$ we have 
$$\Pr(P,G) = \frac{1}{|P|} \sum_{y \in P} \frac{|C_G(y)|}{|G|} =  \frac{1}{|P|} \sum_{y \in P} \frac{1}{|y^G|}.$$

Since $|y^G| \le 2 $ for every $y \in P$, it follows that $\pr(P,G) \geq 1/2$. Because of Lemma \ref{subgroups} (2), every Hall $p'$-subgroup $H$ of $G$ satisfies $\pr(P,H) \ge 1/2$. On the other hand, $F(G)=T$ and so $|G:F(G)| = 2^{s}$, which can be arbitrarily large.} $\qed$
\end{example}

Note that in the above example the exponent of $G/F(G)$ is small. 
 This is related to the following lemma.

\begin{lemma} \label{Hall2-switch} Let $G$ be a soluble group and assume that for every $p\in \pi(T)$ there is a Sylow $p$-subgroup $P$ of $T$ and a Hall $p'$-subgroup $H$ of $G$ such that $\pr(P, H) \geq \epsilon > 0$. Then the exponent of $G/F(G)$ is $\epsilon$-bounded.  
\end{lemma}

\begin{proof} Let $P$ be a Sylow $p$-subgroup of $T$ and $H$ a Hall $p'$-subgroup of $G$ such that $\text{Pr}(P, H) \geq \ep$. Choose $q\in\pi(H)$. It follows from Lemma \ref{subgroups} (2) that $\pr(P,Q) \ge \epsilon$ for some Sylow $q$-subgroup $Q$ of $H$, which is also a Sylow $q$-subgroup of $G$. Therefore $\pr^*(T, G) \ge \epsilon$. Thus, by Theorem \ref{expG/F(G)}, the exponent of $G/F(G)$ is $\epsilon$-bounded. 
\end{proof}

Observe that Example \ref{exem1} also shows that under the hypothesis of Theorem \ref{newTh}  the index of $F(G)$ in $G$ can be arbitrarily large.

\section{Theorem  \ref{newTh}}

A proof of the next proposition can be found in \cite[Proposition 1.2]{ds22}. 

\begin{proposition} \label{prop1.2}
    Let $K$ be a subgroup of a group $G$, and let $\epsilon$ be a positive number such that $\pr(K,G) \ge \epsilon$. Then there is a normal subgroup $H \le G$ and a subgroup $B \le K$ such that the indices $|G:H|$ and $|K: B|$, and the order of the commutator subgroup $[H,B]$ are $\epsilon$-bounded.
\end{proposition}

Our proof of Theorem \ref{newTh} also requires the following result from \cite{detomi2021finite}.

\begin{theorem}\label{Th3.1}
Let $G$ be a     group such that $|x^{\gamma_k(G)}| \le n$ for any $x \in G$. Then $\gamma_{k+1}(G)$ has     $(k,n)$-bounded order.    
\end{theorem}
\medskip

Now we are ready to prove Theorem \ref{newTh}, which we restate here for the reader's convenience: \medskip

  { \it Let $G$ be a group and $T=\gamma_k(G)$ for some $k\ge1$. Suppose that $\Pr(P,G)\ge\epsilon>0$ for every Sylow subgroup $P$ of $T$. Then $G$ has a nilpotent normal subgroup $R$ of nilpotency class at most $k+1$ such that both the exponent of $G/R$ and the order of $\gamma_{k+1}(R)$ are $\epsilon$-bounded. }

\begin{proof}
Whenever $P$ is a Sylow subgroup of $T$ and $Q$ is a Sylow subgroup of $G$, by Lemma \ref{subgroups} (2) we have that $\pr(P,Q) \ge \epsilon$. Thus, by Theorem \ref{expG/F(G)} the exponent of $G/F(G)$ is $\epsilon$-bounded. Set $F = F(G)$. For each $p\in\pi(F)$ let 
 $S_p$ be the Sylow $p$-subgroup of $F$ and $U_p=T\cap S_p$.
 
Again,  it follows from Lemma \ref{subgroups} (2)  that $\pr(U_p, G) \ge \epsilon$. Hence, Proposition \ref{prop1.2} tells us that there is a normal subgroup $H \le G$ and a subgroup $B_p\le U_p$ such that the indices $|G:H|$ and $|U_p:B_p|$ and the order of the commutator subgroup $[H,B_p]$ are $\epsilon$-bounded. Set $H_p=H\cap S_p$. Clearly, $|S_p:H_p|$ and the order of $[H_p,B_p]$ are $\epsilon$-bounded. In particular, $|x^{B_p}|$ is $\epsilon$-bounded for every $x\in H_p$. Since $B_p$ has $\epsilon$-bounded index in $U_p$ and $\gamma_k(S_p) \le U_p$, we deduce that $|x^{\gamma_k(S_p)}|$ is $\epsilon$-bounded for every $x \in H_p$. Now, Theorem \ref{Th3.1} implies that $\gamma_{k+1}(H_p)$ has $\epsilon$-bounded order. 
Set $L_p = C_{H_p}(\gamma_{k+1}(H_p))$, and note that $|H_p:L_p|$ is $\epsilon$-bounded. Since $H$ and $S_p$ are normal in $G$, so is $L_p$. Therefore, $L_p$ is a normal nilpotent subgroup of $G$ with nilpotency class at most $k+1$ such that $|S_p: L_p|$  and the order of $\gamma_{k+1}(L_p)$ are $\epsilon$-bounded.

Since the bounds on $|S_p:L_p|$ and $|\gamma_{k+1}(L_p)|$ do not depend on $p$, it follows that there is an $\epsilon$-bounded constant $C$ such that $S_p = L_p$ and $\gamma_{k+1}(L_p) = 1$ whenever $p \ge C$. 
 Set $R = \prod_{p \in \pi(F) } L_p $. Thus, all Sylow subgroups of $F/R$ have $\epsilon$-bounded order and are trivial for primes larger than $C$. 
Therefore the index of $R$ in $F$ is $\epsilon$-bounded. Moreover, $R$ is of class at most $k+1$ and $|\gamma_{k+1}(R)|$ is $\epsilon$-bounded. Since $G/F$ has $\epsilon$-bounded exponent, we conclude that $G/R$ has $\epsilon$-bounded exponent. This completes the proof.
\end{proof}

\section{Theorems \ref{genfit} and \ref{genfit2}} 

In this section, we prove the last two theorems, which analyse the influence of the commuting probabilities of Sylow subgroups of $G$ and coprime Sylow subgroups of $F^*(G)$ (resp., $F_2^*(G)$).

\begin{proof}[Proof of Theorem \ref{genfit}] 

Part (1). Recall that $T=F^*(G)$ and $\pr^*(T,G)\ge\epsilon>0$. We need to prove that 
the exponent of $G/F(G)$ is $\epsilon$-bounded.

Write $F^*(G)=EF$, where $E$ is the product of subnormal quasisimple subgroups of $G$ and $F=F(G)$. It follows from Lemma \ref{star} that 
$$\pr^*(E/Z(E), E/Z(E)) \ge \pr^*(T, G) \ge \epsilon .$$

So we can apply Lemma \ref{Eb} to 
$E/Z(E)$ and conclude that it has $\epsilon$-bounded order. By Schur's theorem, $E$ has $\epsilon$-bounded order, and so $C=C_G(E)$ has $\epsilon$-bounded index in $G$. Moreover, observe that $F^*(C) = F$. 
 Hence, it is sufficient to show that if $P$ is a Sylow $p$-subgroup of $C$, then the exponent of $P/O_p(C)$ is $\epsilon$-bounded. 

Let $P_0 = C_P(O_{p'}(F))$. First, we will show that $P/P_0$ has $\epsilon$-bounded exponent. Write 
$$O_{p'}(F) = S_1 \times \cdots \times S_l,$$ 
where $S_i$ is a Sylow $p_i$-subgroup of $O_{p'}(F)$. 
By hypothesis we know that $\pr(P, S_i) \ge \epsilon$. Therefore from Lemma \ref{boundindex} we obtain that $|P : C_P(S_i)|$ is $\epsilon$-bounded. Thus, for any $S_i$ the index of $C_P(S_i)$ in $P$ is $\epsilon$-bounded, say at most $m$. Since $m$ does not depend on $p$ or $i$, it follows that the exponent of $P/P_0$ is $\epsilon$-bounded, as required.

Let $G_1=\langle P_0^C\rangle$ be the normal closure of $P_0$. As $F(G_1) = G_1 \cap F$ we observe that $O_{p'}(F(G_1)) \le O_{p'}(F)$. Hence, $O_{p'}(F(G_1))$ is centralised by $P_0$ and we obtain that $O_{p'}(F(G_1)) \le Z(G_1)$. 

Obviously, we have $F({G_1}/{Z(G_1)}) ={F(G_1)}/{Z(G_1)}$ so we can factor out $Z(G_1)$ and assume that $O_p(G_1)=F(G_1)$.
 
Let $P_1 = F(G_1)$. Since $G_1$ and $P_1$ are normal subgroups of $C$, it follows from Lemma \ref{star} (1) that 
$$\pr^*(P_1,G_1) \ge \pr^*(F,G) \ge \epsilon.$$ 
Thus, for every prime $q \in \pi(G_1)$ distinct from $p$ there is a Sylow $q$-subgroup $Q$ of $G_1$  such that $\pr(P_1,Q) \ge \epsilon$. 
It follows from Lemma \ref{boundindex} that $|Q : C_Q(P_1)|$ is $\epsilon$-bounded.
Keeping in mind that $F = F^*(C)$, we deduce that $P_1 = F^*(G_1)$. This implies that $ C_{G_1}(P_1) \le P_1$, so we must have $C_Q(P_1) = 1$. Hence, $Q$ has $\epsilon$-bounded order. 

Let $P_2$ be a Sylow $p$-subgroup of $G_1$. We have shown that, for any $q \in\pi(G_1)\setminus\{p\}$, every Sylow $q$-subgroup $Q$ of $G_1$ has $\epsilon$-bounded order, say at most $n$. In particular, $Q = 1$ for all primes $q > n$. Since $n$ does not depend on $q$ and is $\epsilon$-bounded, it follows that index of $P_2$ in $G_1$ is $\epsilon$-bounded.
 Moreover, observe that $O_p(G_1) = P_1$ is the normal core of $P_2$, and this implies that $P_1$ has $\epsilon$-bounded index in $G_1$. Therefore, $P/P_1$ has $\epsilon$-bounded exponent. Since $P_1 \le O_p(C)$, it follows that $P/O_p(C)$ has $\epsilon$-bounded exponent, as claimed. 

 We have shown that, for every $p \in \pi(C)$ and every Sylow $p$-subgroup $P$, the exponent of $P/O_p(C)$ is $\epsilon$-bounded, say at most $m$. In particular, $P = O_p(C) \le F(C)$ for all primes $p > m$. Since $m$ does not depend on $p$ and is $\epsilon$-bounded, it follows that the exponent of $C/F$ is $\epsilon$-bounded. Therefore, $G/F(G)$ has $\epsilon$-bounded exponent, as required.

Part (2).
    Recall that $T=F^*(G)$, $\pr^*(T,G)\ge\epsilon>0$ and $s = |\pi(F(G))|$. We need to prove that 
the index $|G:F(G)|$ is $(s, \epsilon)$-bounded. 
We use the same notation as above. Write
$$O_{p'}(F) = S_1 \times \cdots \times S_l,$$ 
where $S_i$ is a Sylow $p_i$-subgroup of $O_{p'}(F)$ and $l \le s$. 

The proof of the first part of the theorem 
  shows that  $|P:C_P(S_i)|$ is $\epsilon$-bounded. Hence, the index $|P:P_0|$ is $(s, \epsilon)$-bounded, where $P_0 = C_P(O_{p'}(F))$. 
  Following the same path as in the proof of Part (1), 
  we conclude that the index $|P: O_p(G)|$ is $(s, \epsilon)$-bounded, say at most $m$. Thus, for every $p > m$ we must have $P = O_p(G) \le F(G)$. Therefore, the index $|G:F(G)|$ is $(s,\epsilon)$-bounded. 
   \end{proof}

%%%%%%%%%% exemple %%%%%%%%%%%%
Let us show that under the hypothesis of Theorem \ref{genfit} the index $|G:F^*_2(G)|$ can be arbitrarily large. 

\begin{example}\label{exem2}   
{\rm Let $S=S_5$ and $A = A_5 \le S$ be the symmetric group and the alternating group of degree 5, respectively. Let $\pi = \{ p_1,\dots,p_s \}$ be a set of distinct primes, each at least 7. For $i=1,\dots,s$ let $C_i$ be the cyclic group of order $p_i$ and set $G_i=C_i\wr S$. Let $G=\prod_iG_i$ be the direct product of $G_i$. Observe that $F^*_2(G_i) = C_i\wr A$ and $F^*(G) = F(G) = O_\pi (G) $. 
Let $P$ be a Sylow $p$-subgroup of $F(G)$ and $Q$ be a Sylow $q$-subgroup of $G$, for some prime $q$ distinct from $p$. 
 We know that 
$$\Pr(P,Q) = \frac{1}{|P|} \sum_{y \in P} \frac{|C_Q(y)|}{|Q|} =  \frac{1}{|P|} \sum_{y \in P} \frac{1}{|y^Q|}.$$

If $q \in \pi$, then $\Pr(P,Q) = 1$. If $q = 2,3$ or $5$, then $|y^Q|\leq8$, with $|y^Q|=8$ happening when $Q \cong D_8$ is the dihedral group of order 8.  So in any case we have  $\pr(P,Q) \ge 1/8$.  Therefore, $\pr^*(F^*(G),G)\geq 1/8$. On the other hand, $|G:F^*_2(G)|=2^s$, which tends to infinity when $s$ does.}
\end{example}
%%%%%%%%%% exemple %%%%%%%%%%%%

\medskip

%%%%%%%%%%%% Theorem 1.5 %%%%%%%%%%%%%%

Now we are ready to prove Theorem \ref{genfit2}.

\begin{proof}[Proof of Theorem \ref{genfit2}] 
Recall that $G$ is a finite group and   $\pr^*(F^*_2(G),G) \ge \epsilon>0$. We need to prove that the index of $|G : F_2(G)|$ is $\epsilon$-bounded. 

Let $F= F_1(G)$. 
 As 
 \[\pr^*(F^*(G), G) \ge \pr^*(F^*_2(G), G) \ge \epsilon,\] 
 it follows from Theorem \ref{genfit} (1) that the exponent of $G/F$ is $\epsilon$-bounded. Whence, the number of primes dividing $|G/F|$ is $\epsilon$-bounded: set $s = |\pi(G/F)|$.

Let  $M$ be the full preimage of $F^*(G/F)$. 
  Note that  $M \le F_2^*(G)$. 
By applying twice Lemma \ref{star} we get 
\[
 \pr^*(M/F, G/F) \ge \pr^*(M, G) \ge \pr^*(F^*_2(G), G) \ge \epsilon. 
 \]
Hence, we can apply Theorem \ref{genfit} (2) to the group $G/F$ and conclude that the index of $F_2(G)/F$ in $G/F$ 
 is $(s, \epsilon)$-bounded. Since $s$ is also $\epsilon$-bounded, the result follows.
\end{proof}

\section{Acknowledgment}  This work was done during a visit of the second and the third authors to the Department of Mathematics of the University of Padova. They thank the department for the hospitality. They also acknowledge financial support from CNPq and, in case of the third author, from the 
 \lq\lq National Group for Algebraic and Geometric Structures, and their Applications\rq\rq (GNSAGA - INDAM). The first author is a member  of GNSAGA (INDAM).

\bibliographystyle{abbrv}

\end{document}